\documentclass[12pt,a4paper]{amsart}
\usepackage{mathrsfs}
 \usepackage{amsfonts}
 \usepackage{amsthm}
 \usepackage[dvips]{graphicx}
 \usepackage{subfigure}
 \usepackage{comment}
 \usepackage{float}
 \usepackage{color}
\usepackage{cases}
\usepackage{psfrag}
\usepackage[colorlinks=true,pdfpagemode=FullScreen]{hyperref}
\usepackage{amscd}
\usepackage{amssymb}
\usepackage{verbatim}
\usepackage{amsmath}
\usepackage[toc,page]{appendix}
\newenvironment{draft}{\color{red}}{}
\newtheorem{thm}{Theorem}[section]
\newtheorem{cor}[thm]{Corollary}
\newtheorem{lem}[thm]{Lemma}
\newtheorem{prop}[thm]{Proposition}

\theoremstyle{definition}
\newtheorem{defn}[thm]{Definition}
\newtheorem{rem}[thm]{Remark}

\newtheorem{ex}[thm]{Example}

\newcommand{\N}{{\mathbb N}}
\newcommand{\Z}{{\mathbb Z}}

\newcommand{\R}{{\mathbb R}}

\newcommand{\dom}{\mathop{\mathrm{dom}}}

\newcommand{\Lip}{\mathop{\mathrm{Lip}}}

\def\bc{\begin{center}}       \def\ec{\end{center}}
\def\be{\begin{equation}}     \def\ee{\end{equation}}
\def\ba{\begin{array}}        \def\ea{\end{array}}
\def\bea{\begin{eqnarray}}            \def\eea{\end{eqnarray}}
\def\beaa{\begin{eqnarray*}}  \def\eeaa{\end{eqnarray*}}
\def\bd{\begin{draft} }      \def\ed{\end{draft}}

\def\p{\partial}

\def\pz{p_0^+}
\def\pf{p_0^-}

\def\mL{\mathcal {F}_{\bar{x}}}
\def\F1{\mathcal{F}}
\def\a{\alpha_t}
\def\pt{p_t^-}
\def\sig{\sigma}
\def\ep{\epsilon}
\def\O{\mathcal {O}}

\def\pl{p_-}
\def\pr{p_+}

\def\x{x_0}
\def\y{y_0}
\def\z{z_0}
\def\sig{\sigma}
\def\u{u_i}

\def\p{\partial}

\def\ep{\epsilon}

\def\vp{\varphi}

\begin{document}
\title[Front tracking and iterated minmax]{Front tracking and iterated minmax for Hamilton-Jacobi equation in one space variable}
\author{Qiaoling WEI}
\begin{address}{WEI Qiaoling, Universit\'{e} Paris 7, UFR de Math\'ematiques, B\^atiment Sophie Germain
5, rue Thomas Mann, 75205 Paris Cedex 13
FRANCE} \email{weiqiaoling@math.jussieu.fr}
\end{address}

\maketitle
{ 
\begin{abstract} 
In \cite{WQ2}, the viscosity solution of the Hamilton-Jacobi equation was constructed by an ``iterated minimax'' procedure.
Using Dafermos' front tracking method, we give another proof of this construction in the case of Hamilton-Jacobi equations in one space dimension.
This allows us to get a better understanding in this case of the singularities of the viscosity solution.
\end{abstract}

\section{Introduction}

Consider the Cauchy problem for the Hamilton-Jacobi equations in one space variable
\[\leqno(\hbox{HJ})\quad \left\{
     \begin{array}{ll}
       \p_t u+ H(\p_x u)=0,\quad  \\
       u(0,x)=v(x), \quad x\in \R
     \end{array}
   \right.\]
There are in general no global classical $C^1$ solutions  due to the crossing of characteristics. Different attempts to find proper ``weak solutions'' exist, such as Oleinik's and Kru\"{z}kov's entropy conditions, and explicit solutions constructed by Hopf formula \cite{Hopf} for convex Hamiltonians or initial functions. 

The lack of uniqueness of weak solutions led  M. G. Crandall, L. C. Evans, and P. L. Lions to introduce, in the 1980's, the notion of ``viscosity solution'' \cite{PL,GL}. Viscosity solutions need not be differentiable anywhere, which makes their relationship with the classical crossing of characteristics unclear. However, they possess very general existence, uniqueness and stability properties and, in a large class of ``good'' cases, they coincide with the weak solutions introduced before.

There is also a geometric method of constructing weak solutions for non-convex Hamiltonians, proposed by M. Chaperon, called ``minmax solution'', which, may not necessarily be the viscosity solution. In \cite{WQ2}, the author has shown that, for general Hamilton-Jacobi equations, a limiting process of ``iterated minmax'' instead can lead to the viscosity solution.  In one space variable, this fact can be explained more geometrically by investigating the wave front of the geometric solution. In this paper, we will introduce  the method of ``front tracking'', which was first proposed by C.DAFERMOS (\cite{DA}), to reveal the relation of iterated minmax with viscosity solution, and whence give an alternative perspective to the convergence of iterated minmax to viscosity solutions. As an application, in the last section, we will use the limiting process of iterated minmax to describe the propagation of singularities of viscosity solutions.

\section{Preliminary on minmax solution}
We will briefly introduce the geometric frame work for Hamilton-Jacobi equations. We assume at first that $H$ is $C^2$ and $v$ is $C^2$ with bounded Lipschitz constant, denoted by $\Lip(v)$. The characteristics, or, the Hamiltonian flow of $H$, are given by 
\[\vp^t(x_0,y_0)=(x_0+ t\nabla H(y_0),y_0),\]
and the {\it geometric solution} of the (HJ) equation is defined as
\[L=\bigcup_{t}\{t\}\times \vp^t(dv)\subset \R\times T^*\R\]
where $dv:=\{(x,dv(x)),x\in\R\}$ is the 1-graph of $v$.
Indeed, we can identify $L$ with $\tilde{L}=i(L)\subset T^*(\R\times \R)$ by the  map \[i: \R\times T^*\R\to T^*(\R\times \R),\quad (t,x,p)\mapsto (t,x,-H(p),p)\] A $C^1$ solution of the (HJ) equation, if it exists, is a function $u(t,x)$ whose 1-graph $\{(t,x,\p_t u(t,x),\p_x u(t,x))\}$ coincides with $\tilde{L}$, or equivalently, $(x,\p_x u(t,x))=\vp^t(dv)=:L_t$ for all $t$. In general, due to the crossing of characteristics\footnote{It means precisely the crossing under the projection $\pi: \R\times T^*\R\to \R\times \R$, $(t,x,p)\mapsto (t,x)$.}, the geometric solution may be multi-valued, which prevents the existence of such a $C^1$ solution.

The geometric solution can be generated by a $C^2$ family of functions $S_t: \R\times \R^2\to \R$
\be \label{gf} S_t(x,x_0,y_0):=v(x_0)-tH(y_0)+(x-x_0)y_0\ee
in the sense that, for all $t$,
\[L_t=\vp^t(dv)=\{(x,\p_x S_t(x,x_0,y_0))|\p_{(x_0,y_0)}S_t(x,x_0,y_0)=0\}\]
The function $S(t,x,\cdot):=S_t(x,\cdot): \R\times \R\times \R^2\to \R$ is called a {\it generating family} of the geometric solution $L$.

Note that $L$ is contained in the subset $\Pi_p:=\{|p|\leq \Lip(v)\}$. We can assume that $H$ vanishes outside a neighbourhood of $\Pi_p$ without changing $L$. If $Q$ is the quadratic form in $\R^2$ defined by $Q(x_0,y_0)=-x_0y_0$, then $|\nabla S(t,x,\cdot)-\nabla Q|$ is uniformly bounded on compact subsets of $(t,x)$. We say that $S$ is a generating family {\it quadratic at infinity}. 
Indeed, for each compact subset of $(t,x)$, $S(t,x,\cdot)$ can be made exactly equal to the quadratic form $Q$ outside a compact set by means of a fiberwise diffeomorphism.

For any $C^2$ function $f:X:=\R^k\to \R$ quadratic at infinity, such that $f=Q$ outside a compact set, let $f^c:=\{\eta|f(\eta)\leq c\}$ denote the sub-level set of $f$. Note that for $c$ large enough, the homotopy types of $f^c$,
$f^{-c}$ do not depend on $c$; hence we may denote them as $f^{\infty}$
and $f^{-\infty}$. Suppose the quadratic form $Q$ has Morse index
$\lambda$, taking the coefficient field $\Z_2$, the only non-zero homology groups are $H_{\lambda}(f^{\infty},f^{-\infty};\Z_2)$ and 
$H_{n-{\lambda}}(X\setminus f^{-\infty},X\setminus f^{\infty};\Z_2)$. 

Let $\Xi$ and $\Delta$ be generators of $H_{\lambda}(f^{\infty},f^{-\infty};\Z_2)$ and $H_{n-{\lambda}}(X\setminus f^{-\infty},X\setminus f^{\infty};\Z_2)$ respectively.
\begin{defn} The minmax and maxmin of $f$ are defined by
\beaa \inf\max f&:=&\inf_{[\sigma]=\Xi}\max_{\eta\in |\sigma|}f(\eta)\\
      \sup\min f&:=&\sup_{[\sigma]=\Delta}\min_{\eta\in |\sigma|}f(\eta)\eeaa
where $\sigma$ is a relative cycle and $|\sigma|$ denotes its support. 
We call $\sigma$ a descending (resp.
ascending) cycle if $[\sigma]=\Xi$ (resp. $[\sigma]=\Delta$).
\end{defn}

The minmax and maxmin defined in such a way are equal, see \cite{WQ}, and they are a critical value of $f$.

\begin{prop} Assuming that $v$ is a globally Lipschitz $C^2$ function and $H$ is $C^2$, the minmax 
\be\label{minmax} R_0^t v(x):=\inf\max S_t(x,x_0,y_0)=\inf\max_{(x_0,y_0)}\Big(v(x_0)-tH(y_0)+(x-\x)\y\Big),\ee is a weak solution of the (HJ) equation, i.e., it verifies the equation almost everywhere.
\end{prop}

We remark that since $H$ is independent of $t$, the minmax $R_s^t$ for any $s,t$ depends only on the time difference $t-s$. The choice of the notation $R_s^t$ instead of $R^{t-s}$ is simply for preference.
\bigskip

The generating family quadratic at infinity defined by (\ref{gf}), and hence the minmax, can be extended to the Lipschitz framework, where $v$ is globally Lipschitz and $H$ locally Lipschitz, see \cite{WQ2}. In particular, when $v$ is only Lipschitz, the initial 1-graph $dv$ should be replaced by the {\it enlarged pseudograph } $\p v:=\{(x,p),p\in \p v(x)\}$, where $\p$ is Clarke's generalized derivative. 
\begin{defn} Let $f:\R^k\to \R$ be a Lipschitz function, the Clarke's generalized derivative of $f$ at $x$ is defined by
\[\p f(x):= \hbox{co}\{\lim df(x_n),\,x_n\to x,\,x_n\in \hbox{dom}(df)\}\]
where $\hbox{co}$ denotes the convex envelop. A point $x$ is a critical point of $f$ if $0\in \p f(x)$.
\end{defn}

\begin{ex} For $v(x)=|x|$, $\p v(0)=[-1,1]$, the enlarged pseudograph $\p v$ is obtained by adding a vertical segment to the pseudograph $dv:=\{(x,dv(x)),x\in \hbox{dom}(v)\}$.
\end{ex}
The function $S$ given in (\ref{gf}) is a generating family of the {\it generalized geometric solution} $L=\bigcup_t \{t\}\times \vp_H^t (\p v)$ in the sense that
\beaa L_t=\vp_H^t(\p v)&:=&\{(x,y_0)| y_0\in \p v(x_0), x=x_0+tp, p\in \p H(y_0)\}\\&=&\{(x,\p_x S_t(x,x_0,y_0))|0\in \p_{(x_0,y_0)} S_t(x,x_0,y_0)\},\eeaa
where $\vp_H^t:(x_0,y_0)\mapsto \cup_{p\in \p H(y_0)}(x_0+tp,y_0)$ is the {\it generalized Hamiltonian flow}.
We denote by  $C^{\Lip}(\R)$ the set of globally Lipschitz functions on $\R$, by $\|\p f\|$ the Lipschitz constant of $f\in C^{\Lip}(\R)$, and we denote $|f|_K=\max_{x\in K}|f(x)|$ for any compact $K\subset \R$.

We rename by $R_H^tv(x)$ the minmax function defined in(\ref{minmax}),  unless $H$ is specified.

\begin{prop} [\cite{WQ2}] \label{contis} Assume that $H:\R\to\R$ is locally Lipschitz and $v\in C^{\Lip}(\R)$, then,

$1)$ $R_H^tv(x)$ is a critical value of the Lipschitz map $(x_0,y_0)\mapsto S_t(x,x_0,y_0)$;

$2)$ The minmax function $R_H^tv(x)$ is Lipschitz and satisfies $\|\p( R_H^tv)\|\leq \|\p v\|$;

$3)$ For any $t_1,t_2\geq 0$,
\[|R_H^{t_1}v(x) -  R_H^{t_2}v(x)|\leq
|t_1-t_2||H|_{\{|p|\leq \|\p v\|\}}.\]

$4)$ Let $H^0$ and $H^1$ be two Hamiltonians, then
\[|R_{H^0}^{t}v-R_{H^1}^{t}v|_{C^0}\leq |t||H^0-H^1|_{\{|p|\leq \|\p v\|\}}.\]

$5)$ If $v^0, v^1\in C^{\Lip}(\R)$ and $K$ is a compact set in
$\R$, then there is a compact set $\tilde{K}=\{x:|x|\leq |x|_K+T\|\p H|_{\{|p|\leq C\}}\|\}$ such that \[|R_H^{t} v^0 -R_H^{t} v^1|_{K}\leq
|v^0-v^1|_{\tilde{K}}.\]
where $C=\max_i\|\p v_i\|$, $i=0,1$.
\end{prop}
\bigskip
We recall the definition of viscosity solutions introduced by  M.G.Crandall and P.L.Lions, which is defined, in general for first order partial differential equation, \cite{GL}.
\begin{defn}  A function $u\in C^0\big((0,T)\times \R \big)$ is called a \emph{viscosity subsolution} (resp. \emph{supersolution}) of
 \[\p_t u + H(\p_x u)=0\]
 when it has the following property: for every $\psi\in C^1\big((0,T)\times \R \big)$ and every
 point $(t,x)$ at which $u-\psi$ attains a local maximum (resp. minimum), one has
 \[\p_t \psi + H(\p_x \psi)\leq 0,\quad (\hbox{resp}. \geq 0)\/.\]
 The function $u$ is  a \emph{viscosity solution} if it is both a viscosity subsolution and supersolution.
 \end{defn}

 For any convex Lipschitz funtion $f$, its convex conjugate $f^*$ is defined as
\[f^*:\R\to \bar{\R}=\R\cup\{\infty\},\quad f^*(x):=\sup_y(xy-f(y)).\]

\begin{prop} [\cite{these}]If $H\in C^2$ is convex, and $v\in C^{\Lip}(\R)$, then the minmax is reduced to a minimum: 
\[R_0^tv(x)=\min_{x_0}\big(v(x_0)-tH^*(\frac{x-x_0}{t})\big),\quad t>0.\]
\end{prop}

This is one of the Hopf formulae which defines the viscosity solution of the (HJ) equation. There is another kind of Hopf formulae for convex inital functions $v\in C^{\Lip}(\R)$:
\be\label{hopf}J_0^tv(x)=(v^*+tH)^*(x)= \max_{y_0}(xy_0-(v^*(y_0)+tH(y_0))).\ee
It defines also the viscosity solution of the (HJ) equation. See \cite{HP}.

In the following, we will denote by $J_0^tv(x)$  the viscosity solution of the (HJ) equation with initial function $v$.

\begin{lem} Let $v:\R \to \R$ be a convex Lipschitz function, and $v^*:\R\to
\overline{\R}$ be its convex conjugate, then
\[y\in \p v(x)\Longleftrightarrow v^*(y)=xy-v(x).\]
\end{lem}
\begin{proof} Since $v$ is convex, its generalized derivative $\p
v(x)$ is the usual subderivative,
\beaa \p v(x)&=&\{y:\forall x', f(x')\geq f(x)+y(x'-x)\}\\
&=&\{y: \forall x', xy-f(x)\geq x'y-f(x')\}.\eeaa
We conclude by the definition of the convex conjugate $v^*(y)=\max_{x'}x'y-v(x')$.
\end{proof}

\begin{prop} \label{prop255} We have the relation
\[R_0^tv(x)\leq J_0^tv(x).\]
\end{prop}
\begin{proof}
Recall that a generating family for the minmax is given by
\beaa S_t(x,x_0,y_0)&=& v(x_0)+ xy_0- t H(y_0)-x_0y_0\\
                    &=& (v(x_0)-x_0y_0)+ xy_0-t H(y_0).\eeaa
Let $(\bar{x}_0,\bar{y}_0)$ be a point realizing the minmax:
\[R_0^tv(x)=\inf\max S_t(x,x_0,y_0)=S_t(x,\bar{x}_0,\bar{y}_0)\]
Since the minmax is a critical value for the map $(x_0,y_0)\mapsto S_t(x,x_0,y_0)$,
we have

\[0\in \p_{x_0}S_t(x,\bar{x}_0,\bar{y}_0)\Rightarrow \bar{y}_0\in \p
v(\bar{x}_0)\Leftrightarrow v^*(\bar{y_0})=\bar{x}_0\bar{y}_0-v(\bar{x}_0),\] hence
\[R_0^tv(x)= x\bar{y_0}-v^*(\bar{y}_0)- t H(\bar{y}_0)\]
and
\[R_0^t v(x)\leq \max_{y_0}(xy_0-(v^*(y_0)+ tH(y_0)))=J_0^tv(x).\]

\end{proof}

In general, for convex (concave) initial functions and non-convex Hamiltonians, the minmax function
and the viscosity solution may be different. The phenomena of rarefaction serve as simple counterexamples.

\begin{ex} Consider
\[v(x)=\begin{cases} 6x/5,\quad x\leq 0\\-2x/3,\quad x\geq 0\end{cases},
\quad H(p)=-p^3+p^2+p.\]
The initial function violates the viscosity condition, and there occurs rarefaction.
For $t>0$ small and in a small neighborhood of $x=0$,  $R_0^tv(x)$ is not
viscosity and $R_0^tv(x)<J_0^tv(x)$. See Section 4 for a precise argument.
\end{ex}

\section{Front tracking and iterated minmax}
A notable feature of the viscosity solution is that, by its uniqueness, it possesses a semi-group property:
\[J_0^tv(x)=J_s^t\circ J_0^sv(x),\quad 0\leq s<t.\]
On the contrary, a minmax $R_0^tv(x)$ does not necessarily have this feature. One can always refers to rarefactions as counterexamples. The following proposition tells us that the semi-group property acounts exactly for the difference between the minmax and the viscosity solution.

\begin{prop}[\cite{WQ2}]\label{sm} The minmax $R_0^tv(x)$ is the viscosity solution of the (HJ) equation if and only if it has the semi-group property:
\[R_0^tv(x)=R_s^t\circ R_0^sv(x),\quad 0\leq s<t.\]
\end{prop}

This motivates the construction of {\it iterated minmax}. Fixing a time interval $[0,T]$, let $\zeta_n=\{0=t_0<t_1<\dots<t_n=T\}$ be a subdivision of $[0,T]$. To each $s\in [0,T]$, we associate a number $m(\zeta_n,s)$, depending on $\zeta_n$
:
$$m(\zeta_n,s):= i, \quad \textrm{if}\quad t_i\leq s<t_{i+1}.$$
For simplicity, fixing a subdivision, we may abbreviate $m(\zeta_n,s)$ as $m(n,s)$.
\begin{defn} The {\it iterated minmax solution operator} for the (HJ) equation with  respect to a subdivision
$\zeta_n$ is defined as follows: for $0\leq s'<s \leq T$,
\[R_{H,\zeta_n}^{s',s}:=R_{H}^{t_{m(n,s)},s}\circ \dots\circ
 R_H^{s',t_{m(n,s')+1}}.\] When the
Hamiltonian $H$ is fixed , we may abbreviate our notation
$R_H^{s,t}:=R_H^{t-s}$ as $R_s^t$, and the iterated minmax as
\be\label{iterated}R_{s',n}^s:= R_{s',\zeta_n}^s= R_{t_{m(n,s)}}^{s}\circ\dots R_{s'}^{t_{m_n(s')+1}}.\ee  which
we call a n-step minmax, associated to a subdivision.
\end{defn}

Define the length of $\zeta_n$ by $|\zeta_n|:= \max_i|t_i-t_{i+1}|$.
Suppose that $\{\zeta_n\}_n$ is a sequence of subdivisions of $[0,T]$ such that $|\zeta_n|$ tends to zero as $n$ goes to infinity, and
let $\{R_{0,n}^sv(x)\}_n$ be the corresponding sequence of iterated minmax  for an initial function $v\in C^{\Lip}(\R)$. In \cite{WQ2}, we have shown, for the general Hamilton-Jacobi equation, that any such sequence converges to the viscosity solution. 

For the (HJ) equation in one space variable, using the method of ``front tracking'', which was first proposed by C.DAFERMOS (\cite{DA}), we can understand better the relation between iterated minmax and viscosity solution. 
\bigskip

Let us begin by considering the Riemann problem, with initial functions of
the form
\[v(x)=\begin{cases} p_-x,\quad x\leq 0,\\p_+x,\quad x\geq
0.\end{cases}\] There are two possibilities: if $p_-<p_+$, then $v$
is convex; if $p_->p_+$, $v$ is concave.  If $p_-<p_+$, denote by $(H|_{[p_-,p_+]})^{\smile}$ the convex envelop of $H$ on $[p_-,p_+]$, i.e.
\[(H|_{[p_-,p_+]})^{\smile}(p):=\sup\{h(p)|h\leq H,\,h\,\hbox{convex\,on}
[p_-,p_+]\}\] similarly, if $p_-<p_+$, denote by $(H|_{[p_+,p_-]})^{\frown}$ the concave envelop of $H$ on  $[p_+,p_-]$.
By the breakpoints of a piecewise linear function, we refer to the points where the function is not $C^1$. 

By the Hopf formula (\ref{hopf}), the viscosity solution of the Riemann problem of the (HJ) equation with piecewise linear $H$ can be given explicitely.  We take  for
example the convex case.: if $(H|_{[p_-,p_+]})^{\smile}$ has $m$ breakpoints in $(p_-,p_+)$, denoted by $p_1<\dots<p_m$, and $p_0=p_-$, $p_{m+1}=p_+$, set
\[s_i=\frac{H(p_{i+1})-H(p_i)}{p_{i+1}-p_i},\quad 0\leq i\leq m\]
Then
\[J_0^tv(x)=\begin{cases} xp_0-tH(p_0),\quad x\leq ts_0,\\
xp_{i+1}-tH(p_{i+1}),\quad x\in [ts_i,ts_{i+1}],\,0\leq i\leq m-1,\\
xp_{m+1}-tH(p_{m+1}),\quad x\geq ts_{m}.\end{cases}\]
In particular, $J_0^tv(x)$ has $m+1$ shock $\chi_i(t)=s_it$, $0\leq i\leq m$.

\begin{rem}\label{visconlin} We remark that, by the use of the convex (resp. concave) envelop of $H$,
it follows directly that, at each shock $\chi_i(t)$ of the viscosity solution
$J_0^tv(x)$, with the jump of derivatives $p_i$, $p_{i+1}$, the graph of $H$ between $p_i$ and $p_{i+1}$ lies
above (resp. below) the segment joining $(p_i,H(p_{i}))$ and $(p_{i+1},H(p_{i+1}))$. This is the so-called Oleinik's entropy condition for viscosity solutions,
see Lemma \ref{entropy} in the next section.
\end{rem}

Now we are willing to investigate the minmax solution under the same hypotheses.
We first give a profile for the wave front. The {\it wave front} of the geometric solution at time $t$ is given by \[ \mathcal
{F}^t:=\{(x,S_t(x,x_0,y_0)) | y_0\in \p v(x_0), x\in \x+t\p H(\y)\}.\]
For any subset $A\subset \R$, define
\[\mathcal {F}_A^t:=\mathcal{F}^t|_{\{x_0\in A\}}=\{(x,S_t(x,\z))\in \mathcal{F}^t, \x\in A\}.\]

We claim that the wave front $\mathcal{F}^t$ for $\vp_H^t(\p v)$ with $v$ and $H$ piecewise linear
(with finite pieces) is formed by pieces of straight line segments. Indeed,

1) $\F1_+^t:=\F1_{\{\x>0\}}^t$ and $\F1_-^t:=\F1_{\{\x<0\}}^t$ are two lines with slope
$\pr$ and $\pl$ respectively. Take the case $\x<0$ for example,
one has $y_0=v'(x_0)= \pl$, and
\[\F1_{\{x_0\}}^t=\{z(\x,y)=(\x+ ty,v(\x)+ t (y\pl-H(\pl))): y\in \p H(\pl)\}.\]
Then, for any $\x,\,\x'<0$, $y,\,y'\in\p H(\pl)$, the chord
connecting $z(\x,y)$ and $z(\x',y')$ is of slope $\pl$ . \\

2) Without loss of generality, we assume that
$p_-=p'_0<\dots<p'_k=p_+$ $(or\,p_+=p_0'<\dots<p_k'=p_-)$are the breakpoints of $H$ between $\pl$
and $\pr$. Then $\F1_{\{0\}}^t=\F1^t|_{\{\x=0\}}$ consists of $k$ line
segments with slope $p_i'$ which correspond to the breakpoints
$p'_i$. This can be seen from the formula
\[\F1_{\{0\}}^t=\{z(y,p)=(ty, t(yp-H(p))): p\in [\pr,\pl], y\in \p
H(p)\}.\] If $\pr,\,\pl$ are not breakpoints of $H$, then $\F1_{\{0\}}^t$
loses two line segments of slopes $p_+$ and $p_-$, but the whole
wave front $\F1^t$ does not change in the presence of the two segments
$\F1_{-}^t$ and $\F1_{+}^t$.

\begin{ex} Let $p_-<p_+$, which relates to a convex Riemann
initial function $v$. Figure \ref{hfrown} gives $H$ and its
corresponding wave front. The minmax is obtained by choosing the max.
\end{ex}

\input{hfrown.TpX}

\begin{lem}\label{position} For the Riemann problem with piecewise linear Hamiltonian, suppose that the graph of $H$ between $p_-$ and $p_+$ lies above (resp. below) the segment joining
$(p_-,H(p_-))$ and $(p_+,H(p_+))$. Assuming $p_-<p_+$ (resp. $p_->p_+$), then $\mathcal{F}_{\{0\}}^t$ lies below (resp. above) the graph of the viscosity solution $J_0^tv(x)$ in the wave front.
\end{lem}
\begin{proof} We will give the proof in the case where $p_-<p_+$, while the other case is similar.
For convenience, we may assume that $p_{\pm}$ are not breakpoints of $H$. The viscosity solution $J_0^tv(x)$ has a shock $\chi(t)$ at which $\F1_{-}^t $ and $\F1_+^t$ intersect,
\[\chi(t)=\frac{H(p_-)-H(p_+)}{p_--p_+}t=x_-(t)+tH'(p_-)=x_+(t)+tH'(p_+)\]
for some $x_-(t)<0$ and $x_+(t)>0$. We will show that, at $x=\chi(t)$, $\F1_{\{0\}}^t$ lies below
$(\chi(t), p_-x_-(t)+t(p_-H'(p_-)-H(p_-)))$. The points in $\F1_{\{0\}}^t$ are given by $(ty,t(yp-H(p)))$ with $y\in \p H(p)$ and
$p$ breakpoint of $H$ in $(p_-,p_+)$. Let $ty=\chi(t)=x_-(t)+ tH'(p_-)$,
\beaa && t(y p-H(p))-(p_-x_-(t)+t(p_-H'(p_-)-H(p_-)))\\&=& t(y p-H(p))-t(yp_-H(p_-))\\
&=&t(p-p_-)(y-\frac{H(p)-H(p_-)}{p-p_-})\leq 0,\eeaa
where the inequality comes from $y=\frac{H(p_-)-H(p_+)}{p_--p_+}$ and the hypothesis on the graph of $H$.
Hence $\F1_{\{0\}}^t$ lies below the graph of $J_0^tv(x)$. In particular, they can intersect only at $\chi(t)$.
\end{proof}

\begin{defn}We say that $a\in \R$ admits a descending (resp.ascending)
cycle if there is a descending (resp.ascending) cycle $\sigma$ along
which $a$ is the maximum (resp.minimum) of the generating function
$S$.
\end{defn}
\begin{lem}
\label{rp} If $a\in \R$ admits at the same time a descending cycle
and an ascending cycle, then $a$ is both the minmax and maxmin
value.
\end{lem}

\begin{proof}
By definition, it is easy to see that $\inf\max S\leq a\leq \sup\min
S$. The inverse inequality follows from the fact that a descending
cycle and an ascending cycle must intersect.
\end{proof}

\begin{prop} For the Riemann problem with piecewise linear Hamiltonian, the minmax solution $R_0^t v(x)$
coincides with the viscosity solution.
\end{prop}
\begin{proof} We first remark that for an arbitrary initial function,
the minmax and the viscosity solution may differ immediately.

Consider the Riemann problem with initial value $v(x)=p_-x,\,x\leq 0$,
$v(x)=p_+x,\,x>0$, with $p_-<p_+$. It is sufficient to prove that
$R_0^t v(x)$ is piecewise linear and has $m+1$ shocks
$\chi_i(t)=ts_i$, $0\leq i\leq m$, where
$s_i=\frac{H(p_{i+1})-H(p_i)}{p_{i+1}-p_i}$, with
$p_-=p_0<\dots<p_{m+1}=p_+$ the breakpoints
of the convex envelope $(H|_{[p_-,p_+]})^{\smile}$.

For a fixed time $t$, let $V_i$, $0\leq i\leq m$, be the intersection point of the line
segments corresponding to $p_i$ and $p_{i+1}$ in the wave front
$\F1^t$
\[V_i=(ts_i,t(s_ip_i-H(p_i)))=
(ts_i,t(s_ip_{i+1}-H(p_{i+1}))).\]

We want to show that the minmax $R_0^t v(x)$ is obtained from $\F1$
by selecting the segments $\overline{V_iV_{i+1}}$, \, $0\leq i\leq m-1$,
\[\overline{V_iV_{i+1}}=\{(ts,t(sp_{i+1}-H(p_{i+1})))
:s\,\textrm{between}\,s_i\,\textrm{and}\,s_{i+1}\} .\] Fix any $i$, $0\leq i\leq m-1$,
note $c_i^s:= t(sp_{i+1}-H(p_{i+1}))$. We claim that
$(ts,c_i^s)$ admits a descending simplex, i.e. there exists a
descending simplex $\sig_i^s$ such that
\[\max_{\z\in \sig_i^s}S_t(ts,\z)=c_i^s.\]
Write explicitly
 \beaa S_t(t s,\z)-c_i^s&=& \pr
\x-t H(\y)+ (t s-\x)\y - t(sp_{i+1}-H(p_{i+1}))\\
&=&\x(\pr-\y)+ t A_{i+1,s}(\y),\quad \hbox{if}\,\,\x\geq 0.\\S_t(t
s,\z)-c_i^s&=&\x(\pl-\y)+ tA_{i+1,s}(\y),\quad \hbox{if}\,\,\x\leq 0,\eeaa where
\beaa
A_{i+1,s}(\y):=\begin{cases}(\y-p_{i+1})\left(s-\frac{H(\y)-H(p_{i+1})}{\y-p_{i+1}}\right),\quad
\textrm{for}\quad \y\neq p_{i+1},\\0,\quad
\textrm{otherwise}.\end{cases}\eeaa By the convexity  of $(H|_{[p_-,p_+]})^{\smile}$, we
get
\[A_{i+1,s}(\y)\leq 0,\quad\hbox{if}\,\, y_0\in[p_-,p_+],\,s\,\textrm{between\,$s_i$\,and $s_{i+1}$}.\]
Note that the minmax and maxmin depend on the Hamiltonian $H$ only
on the compact set $[p_-,p_+]$. For each fixed $i$, we can modify $H$ outside $[p_-,p_+]$ by requiring
 \[ \frac{H(\y)-H(p_{i+1})}{\y-p_{i+1}}=\left\{
                                               \begin{array}{ll}
                                              \frac{H(p_-)-H(p_{i+1})}{p_--p_{i+1}}, & y_0\leq p_-, \\
                                              \frac{H(p_+)-H(p_{i+1})}{p_+-p_{i+1}}, & y_0\geq p_+,
                                              \end{array}
                                              \right.\]
 so that
\[A_{i+1,s}(y_0)\leq 0,\quad \forall y_0\in \R.\]
We construct a descending simplex $\sig_i^s=(\x(\y),\y)$ as follows:
\[
\x(\y)=\begin{cases}\frac{tA_{i+1,s}(\y)}{\y-p_-}+ \theta(y_0),\quad \y\in [p_{i+1},\infty),\\
    \frac{tA_{i+1,s}(\y)}{\y-p_+}-\theta(y_0),\quad \y\in (-\infty,p_{i+1},]
     \end{cases}\]
where $\theta:\R \to \R$ is a non negative continuous function, with
\[\theta(y_0)=\begin{cases} 0,\quad y_0\in [p_-,p_+],\\|y_0|,\quad |y_0|\geq
C,\end{cases}\] with $C>0$ a large constant. One sees that
$\sigma_i^s$ is a descending simplex since the first term in
$x_0(y_0)$ is bounded, thus
\[\lim_{|y_0|\to\infty}|y_0|^{-1}((x_0(y_0),y_0)-(y_0,y_0))=0.\]
Furthermore,
\[b_i^s(y_0):=S_t(ts,(x_0(y_0),y_0))-c_i^s=0,\quad \hbox{if}\,\,y_0\in [p_-,p_+]\]
For $y_0\in (p_+,\infty)$, the first term in $x_0(y_0)$ is negative, and one can verify that
$b_i^s(y_0)$ is negative both in the cases where $x_0(y_0)\leq 0$ and $x_0(y_0)\geq 0$. Similarly, we can show that
$b_i^s(y_0)\leq 0$ for $y_0\in (-\infty,p_-)$. Therefore $c_i^s$ admits $\sigma_i^s$ as a descending simplex.

On the other hand, replacing $\theta$ by $-\theta$ will give us an
ascending simplex for $c_i^s$, i.e. $b_i^s(y_0)\geq 0$, with equality for $y_0\in [p_-,p_+]$. Hence $c_i^s$ is at the same time a
minmax and maxmin value by Lemma \ref{rp}.

For the case where $p_->p_+$, that is, when $v$ is concave, we should take the concave envelop $(H|_{[p_+,p_+]})^{\frown}$ and the proof is similar.
\end{proof}

Now suppose that $v$ is piecewise linear continuous (with finite
pieces). Thanks to its local nature and semi-group property, one can construct the corresponding viscosity solution by
viewing $v$ as a combination of Riemann initial data and chasing the
interactions between the shocks of each sub Riemann problem. This is the so-called {\it front tracking} method, which consists in
the following inductive procedure: every time there are collisions between the
shocks, we restart by considering the resulting state as a new
initial function and propagate until the next time of collision. In
the presence of finite propagating speed of characteristics, the
number of shocks of the consequent solution will never blow up.
\begin{prop}[\cite{HL},Lemma 2.6]\label{vis}
Assume that $v(x)$ is a
piecewise linear continuous function with a finite number of
discontinuities of $dv(x)$, and $H$ is a piecewise linear continuous
function with a finite number of breakpoints in $\{|p|\leq \|\p
v\|\}$. Then the viscosity solution of the (HJ) equation $u(t,x)$ is
piecewise linear in $x$ for each $t$, and $\p_x u(t,x)$ takes values
in the finite set
\[\{dv(x),x\in \hbox{dom}(dv)\}\bigcup \{breakpoints\, of \,H\}.\]
Furthermore, the number of shocks of $u$ is finite, bounded by a
constant uniformly for all $t$, and the number of collisions of shocks is finite.
\end{prop}


\begin{lem} \label{viscon} Suppose that $t>0$, and $\mathcal{O}\subset \R$ is an open subset. If for any
$\ep>0$, there exists $s\in (0,\ep)$ such that $R_s^{\tau}\circ R_0^sv(x)$
 is a viscosity solution
of the $(HJ)$ equation for $(\tau,x)\in (s,t)\times \mathcal{O}$, then $R_0^{\tau}v(x)$ is also a viscosity solution for
$(\tau,x)\in (0,t)\times \mathcal{O}$.
\end{lem}
\begin{proof} We will apply the standard argument for the stability of viscosity solutions.
Let $s_n>0$ be a sequence of decreasing numbers such that $s_n\to 0$, and note
\[u_n(\tau,x):=R_{s_n}^{\tau}\circ R_0^{s_n}v(x),\quad (\tau,x)\in \O_{s_n}:=(s_n,t)\times \O\]
and $u(\tau,x):=R_0^{\tau}v(x)$.

Given any $\bar{y}:=(\bar{\tau},\bar{x})\in \O_0$, let $\psi\in C^1(\O_{\bar{\tau}/2})$ such that $u-\psi$ attains a local maximum at $(\bar{\tau},\bar{x})$.
Take $\psi'\in C^1(\O_{\bar{\tau}/2})$ such that $0\leq \psi'<1$ if
$(\tau,x)=:y\neq \bar{y}$ and $\psi'(\bar{y})=1$. Then
$u-(\psi-\psi')$ attains a strict local maximum at $\bar{y}$.
For n large enough, $u_n$ is well-defined in $\O_{\bar{\tau}/2}$, and by Proposition \ref{contis}, $u_n\to u$ on $\overline{\O_{\bar{\tau}/2}}$, thus there exists
$y_n\in \O_{\bar{\tau}/2}$ such that
$u_n-(\psi-\psi')$ attains a local maximum at $y_n$ and $y_n\to \bar{y}$. By assumption, $u_n$ are viscosity solutions,
hence
\[\p_t (\psi-\psi')(y_n)+ H(\p_x (\psi-\psi')(y_n))\leq 0,\]
and we conclude that
\[\p_t \psi(\bar{y})+ H(\p_x \psi(\bar{y}))\leq 0,\]
since $d\psi'(\bar{y})=(\p_t\psi',\p_x\psi')(\bar{y})=0$.
Thus we have proved that $u$ is a viscosity subsolution. Similarly, we can prove that it is a viscosity supersolution, hence a viscosity solution. 
\end{proof}

\begin{prop} \label{274} Assume that $H$ and $v$ satisfy the hypotheses in Proposition \ref{vis}. If $T_1>0$ is such  that for $t\in (0,T_1)$, there is no collision of shocks of the
viscosity solution of the (HJ) problem, then the minmax solution $R_0^tv(x)$ coincides with the viscosity
solution for $t\in (0,T_1]$.
\end{prop}
\begin{proof}  Let $\{x_i\}_{1\leq i\leq n_1}$ be the set discontinuities of
$dv$, and $\chi_j(t)$, $1\leq j\leq n_2$ be the shocks of the
viscosity solution $J_0^tv(x)$ for $t\in (0,T_1)$, where $\chi_j(t)<\chi_{j+1}(t)$. Denote by
$\mathcal{F}^t(v)$ the wave front at time $t$ originated from $v$.
Let $\ep>0$ be small enough such that the sub Riemann problems of
initial function $v$ are independent in the sense that those
$\F1_{\{x_i\}}^t(v)$ emitted from $\p v(x_i)$ in the wave front have no
intersections with each other for $t\in (0,\ep)$. Then for $t\in
(0,\ep)$, by the local property of minmax solution, $R_0^tv(x)$ is
chosen simply by combining every sub minmax problem, hence it is
piecewise linear, and by the local property of viscosity solution,
it is viscosity. For any $s\in (0,\ep)$, writing $v_s(x):=R_0^sv(x)$, we claim that $u_{s,t}(x):=R_s^t\circ R_0^sv(x)=R_s^tv_s(x)$ is the viscosity solution
for $t\in (s, T_1)$. For $t-s<\ep_1$ small enough, when there is no intersections between the $\F1_{\{\chi_j(s)\}}^{t-s}(v_s)$, $1\leq j\leq n_2$, we can apply the same argument as before : the minmax $u_{s,t}(x)$, which is obtained by taking the minmax of each sub Riemann problem with the singularity $\chi_j(s)$, is the viscosity solution ; In particular, by Remark \ref{visconlin} and Lemma \ref{position}, every $\F1_{\{\chi_j(s)\}}^{t-s}(v_s)$ lies above or below the graph of $u_{s,t}(x)$. As long as there is no collision of shocks,  $\F1_{\{\chi_j(s)\}}^{t-s}(v_s)$ and $\F1_{\{\chi_{j+1}(s)\}}^{t-s}(v_s)$ can intersect only strictly above or below the graph of $u_{s,t}(x)$. As the minmax $u_s(t,x)$ is continuous in $t$, it cannot instantaneously jump to the intersection, hence it preserves the shocks $\chi_j(t)$, thus is the viscosity solution.

This proves that, for $\ep>0$ small enough, and any $s\in (0,\ep)$, $R_s^t\circ R_0^sv(x)$, $t\in (s,T_1)$ is the viscosity solution. By  Lemma \ref{viscon}, as $s\to 0$, $R_s^t\circ R_0^sv(x)\to R_0^tv(x)$ which is the viscosity solution for $t\in (0,T_1]$.

\end{proof}

\begin{cor} \label{cor276} Assume that $H$ and $v$ satisfy the hypotheses in Proposition \ref{vis}. Then there exist $k>0$ and $T_j$, $1\leq j\leq k$ with $0<T_1<T_2<\cdots<T_k$ such that,
\be\label{finite}J_0^tv(x)=R_{T_{j}}^{t}\circ\cdots \circ R_{T_0}^{T_1}v(x),\quad T_{j}\leq t\leq T_{j+1},\quad 0\leq j\leq k.\ee
where, by convention,we write $T_0=0$ and $T_{k+1}=+\infty$. Furthermore, for any sequence of subdivisions $\{\zeta_n\}_n$ such that $|\zeta_n|\to 0$, the sequence of iterated minmax $\{R_{0,\zeta_n}^{t}v(x)\}_n$ converges to the viscosity solution
\end{cor}
\begin{proof} Let $T_j$, $1\leq j\leq k$ be the times of collisions of shocks of $J_0^tv(x)$. By Proposition \ref{vis}, $k$ is finite and we obtain (\ref{finite}) by Proposition \ref{274}.
For the second part, given a $T>T_k$, let $\zeta=\{1<t_1<t_2<\cdots<t_n=T\}$ be a subdivision of $[0,T]$; we
will show that \be\label{es276}|R_{0,\zeta_n}^t v-J_0^tv|_{C^0}\leq 2k|H|_{\{|p|\leq \|\p v\|\}}|\zeta|,\quad t\in [0,T].\ee
Denote by $\zeta'$  the subdivision obtained by adding to $\zeta$ the $T_j$, $1\leq j\leq k$, as new dividing points.
Then by the semi-group property of the viscosity solution, we have $R_{0,\zeta'}^tv(x)=J_0^tv(x)$.
Let $T_j^-:=\max_{1\leq i\leq n}\{t_i,\,t_i\leq T_j\}$ and notice that for $|\zeta|$ small, we have $T_j<T_{j+1}^-$.  For $t\in [0,T_1]$, $R_{0,\zeta}^tv(x)=J_0^tv(x)$. Writing $v_j(x):=R_{0,\zeta}^{T_{j}}v(x)$ and $v_j'(x):=R_{0,\zeta'}^{T_{j}}v(x)$, then for $t\in [T_{j+1},T_{j+2}]$,
\beaa &&|R_{0,\zeta}^tv-R_{0,\zeta'}^tv|_{C^0}\\&=& |R_{T_{j+1}^-,\zeta}^t\circ R_{T_j}^{T_{j+1}^-}v_j
-R_{T_{j+1},\zeta}^{t}\circ R_{T_{j+1}^-}^{T_{j+1}}\circ R_{T_j}^{T_{j+1}^-}v_j'|_{C^0}\\
&\leq & |R_{T_{j+1}^-,\zeta}^t\circ R_{T_j}^{T_{j+1}^-}v_j-R_{T_{j+1}^-,\zeta}^t\circ R_{T_{j+1}^-}^{T_{j+1}}\circ R_{T_j}^{T_{j+1}^-}v_j'|_{C^0}\\&+&
|R_{T_{j+1}^-,\zeta}^t\circ R_{T_{j+1}^-}^{T_{j+1}}\circ R_{T_j}^{T_{j+1}^-}v_j'-R_{T_{j+1},\zeta}^{t}\circ R_{T_{j+1}^-}^{T_{j+1}}\circ R_{T_j}^{T_{j+1}^-}v_j'|_{C^0}\\
&\leq & |R_{T_j}^{T_{j+1}^-}v_j-R_{T_{j+1}^-}^{T_{j+1}}\circ R_{T_j}^{T_{j+1}^-}v_j'|_{C^0}+ |T_{j+1}^- -T_{j+1}||H|_{\{|p|\leq \|\p v\|\}}\\
&\leq &|v_j-v_j'|_{C^0}+ 2|T_{j+1}^--T_{j+1}||H|_{\{|p|\leq \|\p v\|\}}.
\eeaa
Hence, by induction on $j$, we can obtain (\ref{es276}). It follows that, if $\{\zeta_n\}_n$ is a sequence of
subdivisions of $[0,T]$, such that $|\zeta_n|\to 0$, we have
\[\lim_{n\to \infty} R_{0,\zeta_n}^t v(x)=J_0^tv(x).\]
\end{proof}

\begin{thm}\label{dim1}
Suppose $v:\R\to\R$ is globally Lipschitz and $H:\R\to\R$ is
locally Lipschitz, then for any sequence of subdivisions $\{\zeta_n\}_n$ of $[0,T]$ such that $|\zeta_n|\to 0$, the sequence of iterated minmax $\{R_{H,\zeta_n}^{0,t}v(x)\}_n$ converges uniformly on compact subsets to the viscosity solution of the $(HJ)$ equation
\[\left\{
    \begin{array}{ll}
      \p_t u + H(\p_x u)=0,\quad  t\in (0,T] \\
      u(0,x)=v(x),\quad x\in \R
    \end{array}
  \right.\]
\end{thm}
\begin{proof} For any compact subset $K\subset \R$, fixed once for
all, we
can modify $H$ and $v$ so that $H$ is linear outside $\mathcal{K}:=\{|p|\leq \|\p
v\|\}$ and $v$ is linear outside $\tilde{K}=\{|x_0|\leq |x|_K+ T\|\p (H|_{\mathcal{K}})\|\}$. For $k\in \N$, take the piecewise linear approximations of $v^k$ and $H^k$ for $v$ and
$H$ such that
\[|v-v^k|_{C^0}\leq \|\p v\|/k,\quad |H-H^k|_{C^0}\leq \|\p H\|/k.\]
Since $H$ and $v$ are linear outside compact sets, we can assume that both $v^k$ and $H^k$ have finite pieces.

Given any sequence of subdivisions $\{\zeta_n\}$ of $[0,T]$ such that
$|\zeta_n|\to 0$, let $R_{H,n}^t v(x)$ be the corresponding iterated minmax. It follows from Propostion \ref{contis} that the sequence $u_n(t,x):=R_{H,n}^tv(x)$ is equi-Lipschitz and uniformly bounded for $(t,x)\in [0,T]\times K$. Hence by the Arzela-Ascoli Theorem, $\{u_n\}_n$ has some convergent subsequence. We denote by   $\bar{R}_H^tv(x)$ its limit. 
We have
\[ |R_{H,n}^t v- R_{H^k,n}^t v^k|_{K}\leq
|v-v^k|_{\tilde{K}}+ T|H-H^k|_{\mathcal {K}}.\] By Corollary \ref{cor276}, for each $k$, the sequence $\{R_{H_k,n}^tv_k\}_n$ converges as $n\to\infty$; denoting its limit by $\bar{R}_{H^k}^tv_k(x)$, it is a viscosity solution of the equation $\p_t u+ H^k(\p_x u)=0$. We have
\[ |\bar{R}_H^t v- \bar{R}_{H^k}^t v^k|_{K}\leq
|v-v^k|_{\tilde{K}}+ T|H-H^k|_{\mathcal {K}}\to 0,\quad k\to\infty.\]
We get from the stability of viscosity solutions (ref.Theorem 1.4 \cite{GL}),
that $\bar{R}_H^t v(x)$ is a viscosity solution for $\p_t u + H(\p_x
u)=0$. By the uniqueness of the viscosity solution for the Cauchy problem, the limit $\bar{R}_H^tv(x)$ is independent of the convergent subsequence. Thus we conclude that the sequence of iterated minmax $\{R_{H,n}^tv(x)\}_n$ converges to the viscosity solution.
\end{proof}

\bigskip

\section{Singularities of viscosity solution}
In this section, we will see how their interpretation as a
limit of minmax help to describe the propogation of singularities of viscosity solutions.

The viscosity solution of the (HJ) equation is characterized equivalently by Oleinik's entropy condition:

Assume that $\mathcal {O}\subset (0,\infty)\times \R$ is an open
set, $u\in C(\mathcal {O})$, and that there is a $C^1$ curve $\chi$ dividing
 $\mathcal {O}$ into two open subsets $\mathcal {O}^+$ and
$\mathcal {O}^-$, $\mathcal {O}=\mathcal {O}^+ \cup \chi \cup
\mathcal {O}^+$.

\begin{lem}[Entropy condition.]\label{entropy}
Let $u \in C(\mathcal {O})$ and $u=u^+$ in $\mathcal {O}^+\cup\chi$,
$u=u^-$ in $\mathcal {O}^-\cup\chi$, where $u^\pm \in C^1(\mathcal
{O}^\pm \cup \chi)$. Then $u$ is the viscosity
solution of the equation $(HJ)$ in $\mathcal {O}$ if and only if:
\begin{enumerate}
\item $u^\pm$ are classical solutions in $\mathcal {O}^\pm$ respectively,
\item The graph of $H$ lies below (resp. above) the line segment
joining the points $(p_t^-, H(p_t^-))$ and $(p_t^+, H(p_t^+))$ if $p_t^+<p_t^-$ (resp.$p_t^-<p_t^+$), where $p_t^{\pm}:=\p_x u^{\pm}(t,\chi(t))$.
\end{enumerate}
\end{lem}
In particular, we say that $\chi(t)$ {\it strictly}
verifies the entropy condition if it verifies the entropy condition and the line segment joining $(p^-_t,H(p_t^-))$ and $(p^+_t,H(p^+_t.))$ is not tangent to the graph of $H$.

\begin{rem} Equivalently, the entropy condition can be described, for example, when $ p_t^+<p_t^-$, as
\be \label{RH}\frac{H(w)-H(p_t^-)}{w-p_t^-}\geq \frac{H(p_t^+)-H(p_t^-)}{p_t^+ - p_t^-}(=\dot{\chi}(t))
\ee for any
$w$ between $p_t^+$ and $p_t^-$,  The ``='' between brackets
refers to the Rankine-Hugoniot condition which is a necessary condition for
$u$ to be a weak solution in the sense of distribution.
\end{rem}

In the following, we suppose that the Hamiltonian $H$ is smooth,
with generic critical and inflection points, and that the initial function $v$ is
globally Lipschitzian.

Recall that the {\it wave front} of the $(HJ)$ equation  at time
$t$ is
\[ \F1=\F1^t(v):=\{(x,S_t(x;x_0,y_0)) | y_0\in \p v(x_0), x= \x+tH'(\y)\},\]
where $S_t(x,x_0,y_0)=v(x_0)+xy_0-tH(y_0)-\x\y$ is the generating family. It is a curve in the
plane $\R^2$ having cusps and self-intersection points. Adopting the notions in \cite{Elia}, a wave
front is said {\it typical} if it has no triple self-intersection points
and the sets $C$ and $D$ of cusp and double self-intersection points respectively are
discrete and disjoint. The {\it branches} of the front $\F1$ are the irreducible
components of $\F1\setminus C$. Every branch is the graph of a
partially defined $C^1$ curve. The union in space-time of the fronts
is called the {\it big front},
denoted by $\tilde{\F1}$,
\[\tilde{\F1}=\tilde{\F1}(v):=\{(t,x,S_t(x;x_0,y_0))|y_0\in \p v(x_0), x=x_0+ t
H'(y_0)\}.\] The projection of the self-intersections of
$\tilde{\F1}$ onto the $(t,x)-$plane are called {\it shocks}.

In particular, we distinguish the {\it genuine big front} which is by
definition the restriction of the big front $\tilde{\F1}$ to
$\{x_0\in \dom (dv)\}$, that is to the points such that $v$ is differentiable at $x_0$, with corresponding {\it genuine
branches} and {\it genuine shocks}.



\begin{prop}\label{ge} Suppose $\tilde{\F1}(v)$ is a typical big
front for the $(HJ)$ equation,
\begin{enumerate}
\item If the genuine shocks strictly verify the entropy condition, then within sufficiently small time, the minmax solution preserves the shocks. As a consequence, it coincides with the viscosity solution.
\item On the contrary, if some genuine shock violates the entropy
condition, by this we include rarefactions \footnote{The rarefaction means there is separation of genuine branches.}, then the
minmax solution admits new shocks other than the genuine shocks\footnote{By this, we may include the disappearance of the genuine shocks.}
\end{enumerate}
\end{prop}

\begin{proof}


1) Let $\bar{x}$ be a  singularity of $v$,  $\chi(t)$  the
genuine shock generated from $\bar{x}$, and $u^\pm (t,x)$  the evolution of the two genuine branches whose intersection gives $\chi(t)$. Note $p_t^\pm:=\p_x
u^\pm(t,\chi(t))$. We shall suppose that $p_0^+ < p_0^-$, the other
case being similar. There is $\epsilon>0$, such that  there exists a neighborhood $\mathcal{O}\in (0,\ep)\times \R$ of $\chi(t)$ in the big front $\tilde{\F1}$ that contains exactly the two branches $u^{\pm}$ as well as
$\F1_{\{\bar{x}\}}^t$, where: \be \label{lx} \mathcal
{F}_{\{\bar{x}\}}^t=\{z(p)=\left( \bar{x}+ t H'(p), v(\bar{x})+ t(p H'(p)-
H(p))\right) ,\,p\in [p_0^+,p_0^-]\}.\ee We claim that for $\ep$ small enough, $\mathcal
{F}_{\bar{x}}^t$ does not interfere with the selection of the minmax
solution $R_0^tv(x)$ for $(t,x)\in \mathcal{O}$. For this purpose, it is enough to show that there is no point in
$\F1_{\{\bar{x}\}}^t$ lying below the genuine shock $(\chi(t),u(t,\chi(t)))$, where \beaa
\chi(t) &=& x^\pm_t+ t
H'(p_t^\pm),\\u(t,\chi(t))&=&u^\pm(0,x^\pm_t)+ t (p_t^\pm
H'(p_t^\pm)-H(p_t^\pm )).\eeaa  Suppose $t\mapsto \a\in [\pz,\pf]$ is a $C^1$ function
of $t$ such that \[z(\a)=(\bar{x}+ t H'(\a), v(\bar{x})+ t(\a H'(\a)-H(\a)))\in\mL^t\cap \{x=\chi(t)\} .\] 

Define
\[ A(t)=(u^-(0,x_t^-)+ t
(p_t^- H'(p_t^-)-H(p_t^- )))-(u(0,\bar{x})+ t(\a H'(\a)-H(\a))).\]
Then $A(0)=0$. Note that $\p_x u^-(0,x_t^-)=\p_x u^-(t,\chi(t))$ because they lie in
the same characteristics,  \beaa \dot{A}(t)&=&\pt \dot{x}_t^-+ \pt
H'(\pt)-H(\pt)-\a H'(\a)+ H(\a)+ t(\pt H''(\pt)\dot{p}_t^--\a
H''(\a)\dot{\alpha}_t)
\\&=&\pt (\dot{\chi}(t)-H'(\pt)-t H''(\pt)\dot{p}_t^-)+ \pt
H'(\pt)-H(\pt)-\a H'(\a)+ H(\a)
\\&\,\,&+ t(\pt H''(\pt)\dot{p}_t^-+ \a(H'(\a)-\dot{\chi}(t)))
 \\&=& (\pt-\a)(\dot{\chi}(t)-\frac{H(\a)-H(\pt)}{\a-\pt})
.\eeaa

If $\chi(t)$ strictly verifies the entropy condition, then we have $\alpha_0:=\lim_{t>0,t\to 0} \alpha_t<p_0^{-}$. Indeed, if $\alpha_0=p_0^-$, then
\[ \dot{\chi}(0):=\lim_{t>0,t\to 0}\dot{\chi}(t)=\lim_{t>0,t\to 0}H'(\alpha_t)+tH''(\alpha_t)\dot{\alpha}_t=H'(\alpha_0)=H'(p_0^-),\]
which means that the shock $\chi$ does not strictly verify the entropy condition. Therefore, for sufficiently small time $t$, we can assume that $\alpha_t<p_t^-$, then by the entropy condition (\ref{RH}), we get $\dot{A}(t)<0$, hence $A(t)<0$.
Now the only way to select a
continuous section in $\F1^t|_{\mathcal{O}}$ is choosing the two branches
$u^{\pm}$ with the genuine shock $\chi(t)$. Since the graph of the minmax solution
$\{(x,R_0^tv(x))\}$ is a continuous section in $\F1^t$, we get
\[R_0^tv(x)=\begin{cases} u^{+}(t,x),\quad (t,x)\in \mathcal{O}^+\cup\chi\\
u^-(t,x),\quad (t,x)\in  \mathcal{O}^-\cup\chi \end{cases}\]
by the local property of the minmax.

2) Suppose that a genuine shock violates the entropy condition: if its
two branches $u^{\pm}$ separate, there must be a piece from $\F1_{\{\bar{x}\}}^t$
to compensate in order to get a continuous section; if they intersect, the
minmax solution can not preserve the genuine shock, otherwise
it satisfies the semi-group property, from which, followed by Proposition
\ref{sm}, that it must be the viscosity solution, thus a contradiction.
\end{proof}

\begin{ex} If $H$ is convex (resp. concave), then in the wave fronts, the branches from the singularites of the initial function always lie above (resp. below) the the genuine branches, hence the min (resp. max) solution is the viscosity solution. See Figure \ref{convex}.
\input{convex.TpX}
\end{ex}

\begin{ex} If $H$ is non convex, but the genuine shock
does not violate the entropy condition, the configuration of the wave front are depicted in Figure \ref{nonconvex}.
\input{nonconvex.TpX}
\end{ex}

Before a critical time $t_c$ where the genuine branches violate the entropy condition, the graph of the viscosity solution, composed of the genuine branches, is always contained in each wave front $\F1^t(v)$, hence in the big front $\tilde{\F1}(v)$. After that,
the behavior of the viscosity solution becomes more complicated since
the branches generated by the singularities of the initial function begin to interfere. As a consequence, the viscosity
solution may not be contained in the big front.

\vspace{6pt}

Instead of considering the big front $\cup_{t\geq t_c}\{t\}\times \F1^t(v)$, let us consider \[\tilde{\F1}=\cup_{t\in (0,\ep)}\{t\}\times\F1^{t}(R_0^{t_c}v)\] for $\ep>0$ small.

Let $v_c(x):=R_0^{t_c}v(x)$, suppose that $R_0^tv_c(x)$ has a $C^1$ shock $\chi(t)$ in $(t,x)\in \O_{\ep}:=(0,\ep)\times \O$ and that $\chi(t)$ violates the
entropy condition. Given a subdivision $\{\zeta_n\}$ of $(0,\ep)$, where
$\zeta_n=\{0=t_0<t_1^{(n)}<\dots<t_n^{(n)}=\ep\}$, let $\chi_n(t)$ be a continuous, piecewise $C^1$ shock
of the related iterated n-step minmax $u_n(t,x):=R_{0,n}^tv_c(x)$.
Note $p_t^{\pm}:=\p_x u_n(t,\chi_n(t)\pm)$. By Proposition \ref{ge} and Lemma \ref{viscon}, for $|\zeta_n|$ small enough,
there is violation of the entropy condition: for $t\in (t_k^{(n)},t_{k+1}^{(n)}]$, $\chi_n(t)$ is
constructed by the intersection of two branches, one of which comes
from $\F1_{\{\chi_n(t_k^{(n)})\}}^{t_k^{(n)}}(u_{n}(t_k^{(n)},\cdot))$. Assuming that
\[\chi_n(t)=\chi_n(t_k^{(n)})+ (t-t_k^{(n)})H'(p_t^-),\quad t\in (t_k^{(n)},t_{k+1}^{(n)}],\]
where $p_t$ lies between $p_{t_k^{(n)}}^+$ and $p_{t_k^{(n)}}^-$ for $t\in (t_k^{(n)},t_{k+1}^{(n)}]$,
note that $\chi_n(t)$ satisfies the
Rankine-Hugoniot condition,
\[\dot{\chi}_n(t_k^{(n)})=\frac{H(p^+_{t_k^{(n)}})-H(p^-_{t_k^{(n)}})}{p^+_{t_k^{(n)}}-p^-_{t_k^{(n)}}}=H'(p_{t_k^{(n)}}^-),\,\,0\leq k\leq n-1.\]

Thus the limit $\bar{R}_0^t v_c(x)$ of the sequence of iterated minmax $\{R_{0,n}^t
v_c(x)\}_n$ has a $C^1$ shock $\bar{\chi}(t)$ satisfying
\be\label{vissing} \dot{\bar{\chi}}(t)=\frac{H(p_t^+)-H(p_t^-)}{p_t^+ -p_t^-}=
H'(p_t^-),\quad 0\leq t \leq \epsilon.\ee

By Theorem \ref{dim1}, the limit $\bar{R}_0^tv_c(x)$ is the viscosity solution,
hence (\ref{vissing}) describes the propagation of singularities of the viscosity solution.

\vspace{6pt}

We see that those characteristics which are not contained in the original big front $\cup_{t>t_c}\{t\}\times\F1^t(v)$ are tangent to the shock $\bar{\chi}(t)$, explained by (\ref{vissing}).This kind of shock
is called a {\it contact shock}. The contact shock is a typical phenomenon
due to rarefaction.
With the
limiting process of iterated minmax, we can see how arise the
``exterior'' characteristics: they are indeed emerging from
the singularities of the minmax iterated step by step.

\vspace{6pt}


In their paper \cite{Bif}, S.Izumiya and G.T.Kossioris have studied
the bifurcation of shock waves for the $(HJ)$ equations in the neighborhood of the
critical time $t_c$ for generic Hamiltonians and initial functions.
By using the method of Legendrian unfoldings, they classified
the local pictures of the bifurcations of the viscosity solution
into eight generic cases.
\vspace{6pt}

Here we will not repeat the description of the generic cases but take several concrete examples to explain the limit process of iterated minmax and hence the singularities of
viscosity solutions after the critical time where there are violations of entropy condition for the genuine branches.

\vspace{6pt}

Let $\bar{x}$ be a singularity of $v$, we first figure out how the wave front looks like by
indicating the relative position of $\F1_{\{\bar{x}\}}^t=\F1|_{\{x_0=\bar{x}\}}$ in $\mathcal {F}=\F1^t(v)$. Indeed,
we get from (\ref{lx}) some key rules to characterize $\F1_{\{\bar{x}\}}^t$:
\begin{enumerate}
\item  The cusps in $\F1_{\{\bar{x}\}}^t$ correspond to the inflection points of
$H$; and the convexity of the branches of $\F1_{\{\bar{x}\}}^t$ coincide with that of $H$;

\item The tangency of $\F1_{\{\bar{x}\}}^t$ at $z(p)$ is $p$;

\item  $\F1_{\{\bar{x}\}}^t$ and the two genuine branches $u^{\pm}$
are joined in a $C^1$ smooth manner at the extremities
$z(p_{0}^\pm)$ with tangency $p_{0}^\pm$.
\end{enumerate}

\vspace{6pt}

\begin{ex}[Rarefaction]  Consider
\[v(x)=\begin{cases} -x(x-1),\quad x\leq 0\\x(x-1),\quad x\geq 0\end{cases},\quad H(p)=-p^3+p^2+p.\]
\end{ex}
Look at a neighborhood of the singularity $x=0$ of $v$.

\input{case11.TpX}

For $t>s>0$ small, the
wave fronts are depicted as follows
\begin{figure}[H]
\begin{minipage}[t]{0.5\linewidth}
\begin{flushleft} \large
\includegraphics[width=2.2in]{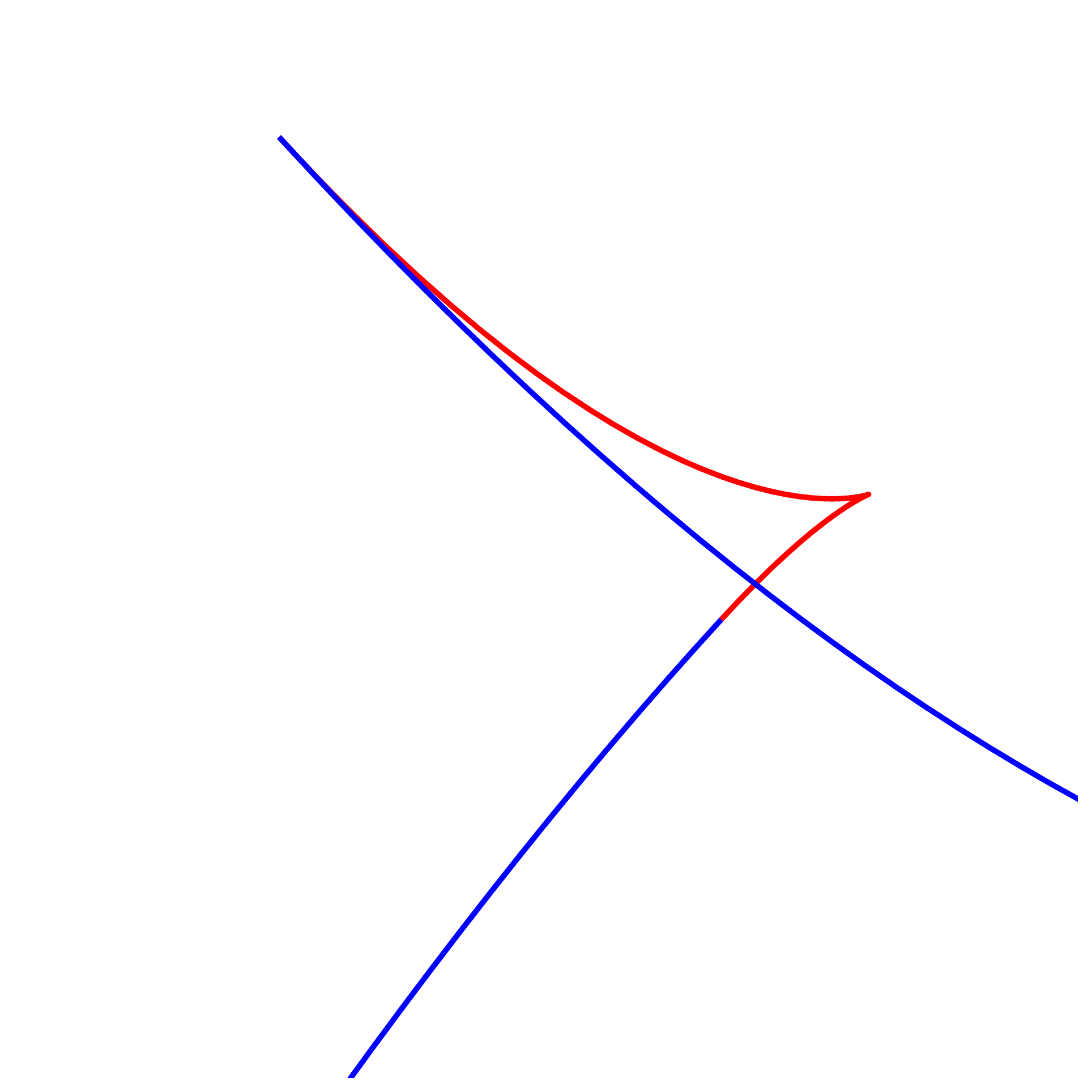}
\caption{$\mathcal{F}^{t}(v)$}\label{m1}
\end{flushleft}
\end{minipage}%
\begin{minipage}[t]{0.5\linewidth}
\begin{flushright} \large
\includegraphics[width=2.2in]{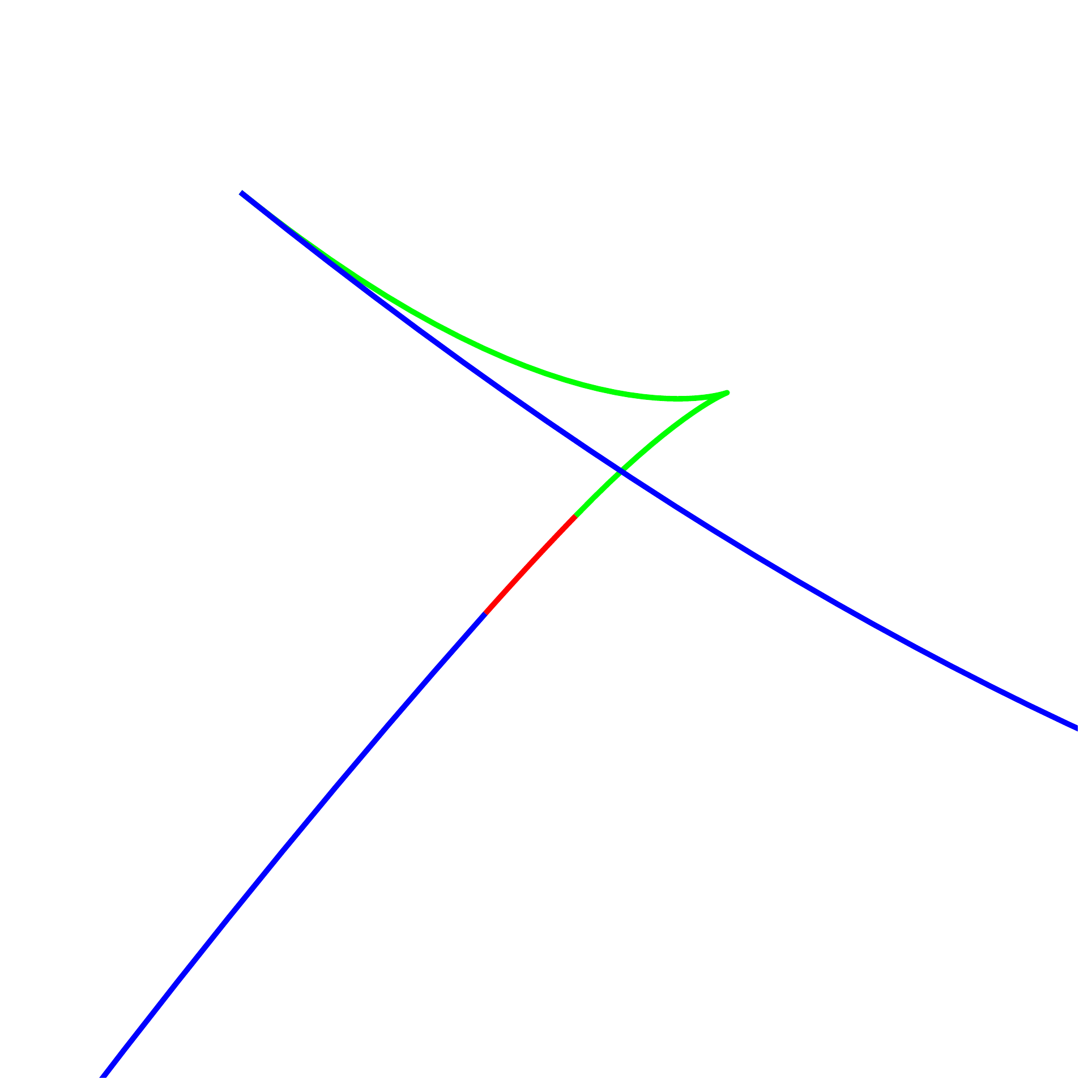}
\caption{$\mathcal{F}^{t-s}(R_0^{s}v)$}\label{m2}
\end{flushright}
\end{minipage}
\end{figure}

\begin{figure}[H]
\begin{minipage}[t]{0.5\linewidth}
\begin{flushleft} \large
\includegraphics[width=2.0in]{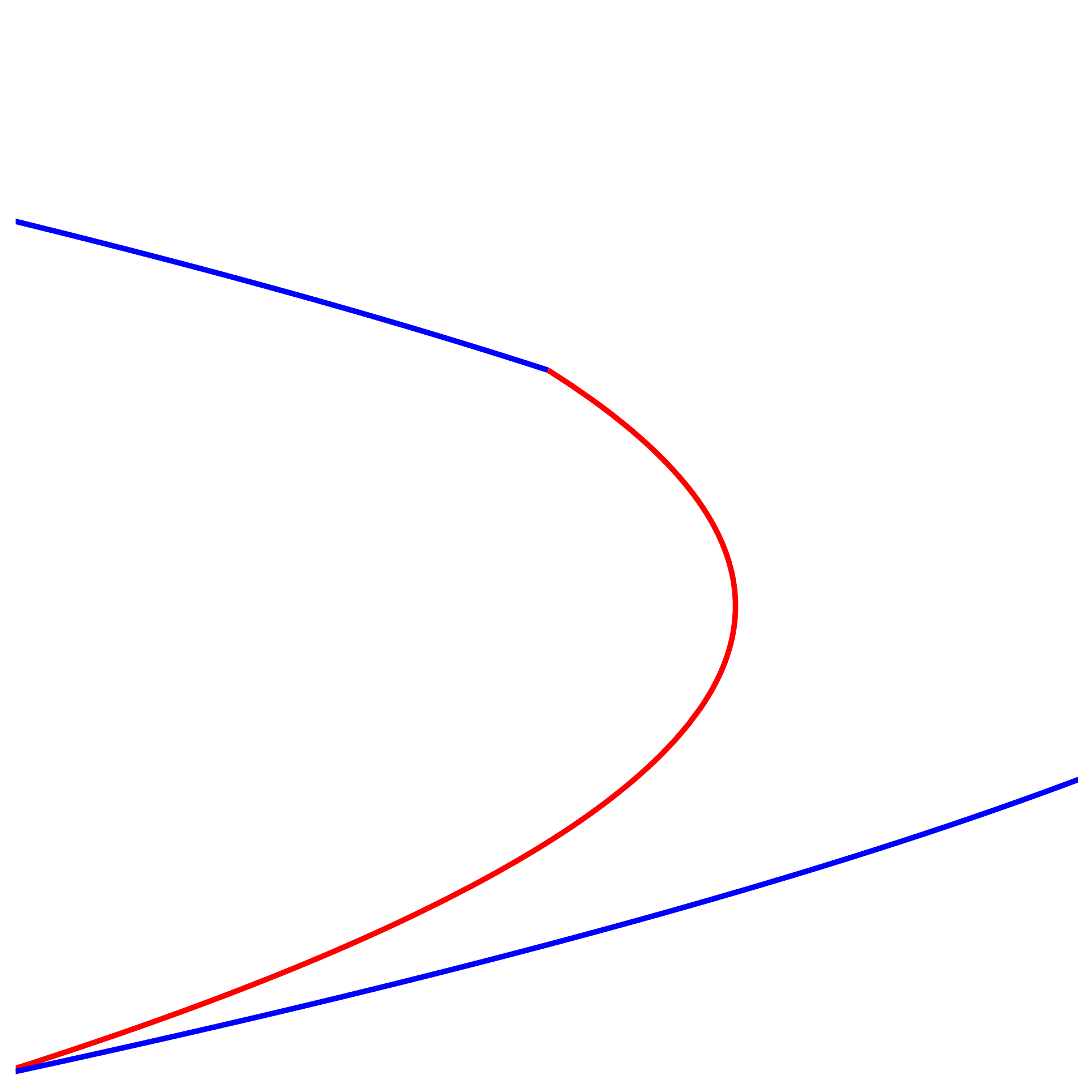}
\caption{$\vp^t(\p v)$}\label{lm1}
\end{flushleft}
\end{minipage}%
\begin{minipage}[t]{0.5\linewidth}
\begin{flushright} \large
\includegraphics[width=2.0in]{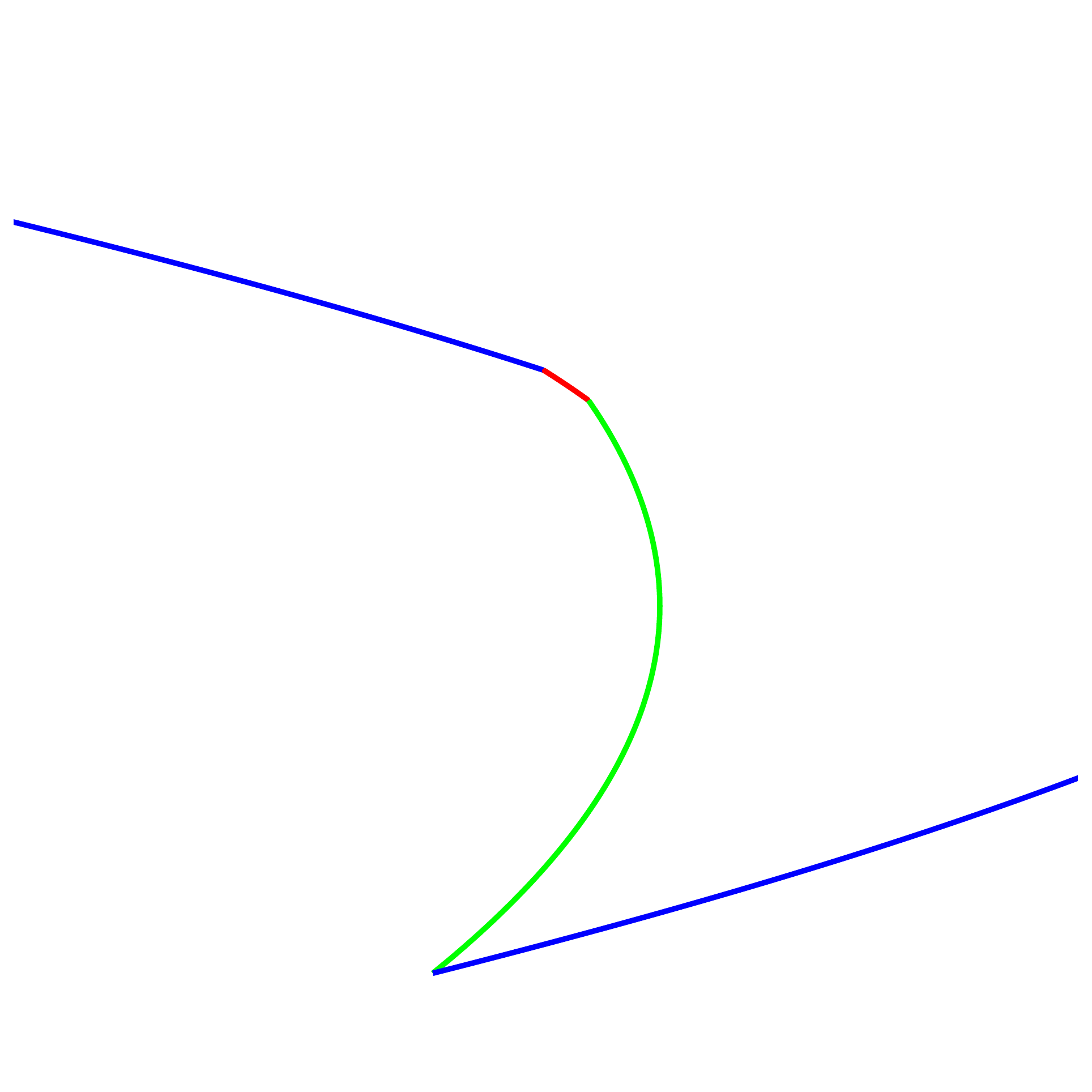}
\caption{$\vp^{t-s}(\p R_0^sv)$}\label{lm2}
\end{flushright}
\end{minipage}
\end{figure}

In the wave front $\F1^t(v)$, the two branches in blue are genuine branches, and the curve in red is $\F1_{\{0\}}^t(v)$.
There is rarefaction, that is, separation of genuine branches. The 1-step minmax
$R_0^tv(x)$ attains the minimum in the wave front, and the shock $\chi_1(t)$ of $R_0^tv(x)$ is given by the intersection
of the right genuine branch and a branch of $\F1_{\{0\}}^t(v)$. Revealed by the difference of $\vp^t(\p v)$ and $\vp^{t-s}(\p R_0^sv)$, the shock $\chi_1(t)$ violates the entropy condition for $t$ small, hence,
as depicted in the wave front $\F1^{t-s}(R_0^sv)$, there is rarefaction for the genuine branches of $R_0^sv$ and the 2-step minmax
$R_s^t\circ R_0^sv(x)$ has a shock $\chi_2(t)$ which, for $t>s$, is given by the intersection of
the right genuine branch of $R_0^sv(x)$ and a section of $\F1_{\{\chi_2(s)\}}^{t-s}(R_0^sv)$.
By the limiting process of iterated minmax, we then get a contact shock as described in the left of Figure \ref{shocktwo}.

\begin{ex} Consider
\[v(x)=\begin{cases} x(x-1),\quad x\geq 0\\x(x+1),\quad x\leq 0\end{cases}
\quad H(p)=p^4-p^2.\]
\end{ex}
Look at a neighborhood of the singularity $x=0$ of $v$:
\vspace{20pt}
\input{case22.TpX}

For $t,s>0$ small, the
wave fronts are depicted as follows
\begin{figure}[H]
\begin{minipage}[t]{0.5\linewidth}
\centering
\includegraphics[width=1.6in]{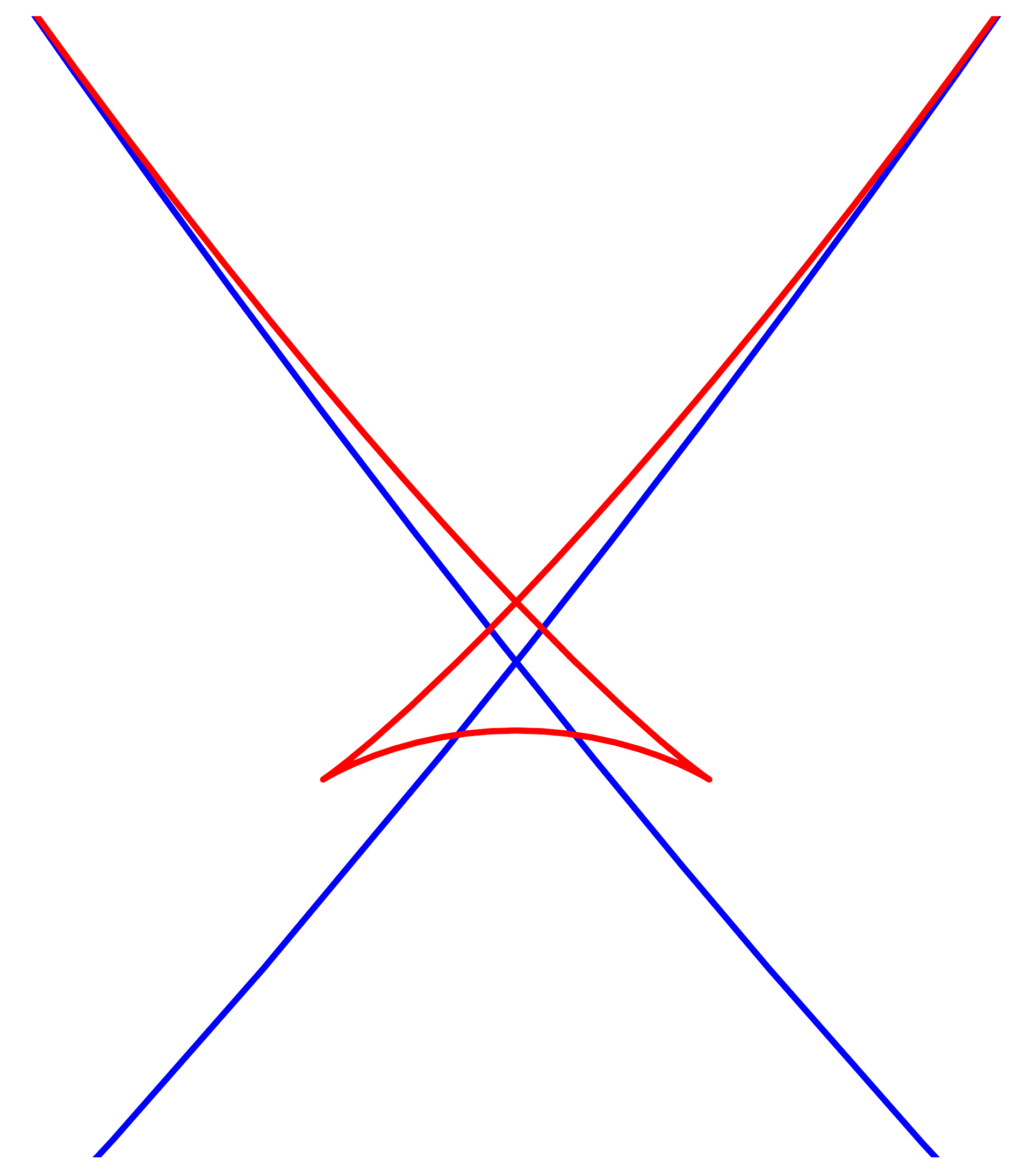}\label{mm2}
\caption{$\F1^t(v)$}
\end{minipage}%
\begin{minipage}[t]{0.5\linewidth}
\centering
\includegraphics[width=2.0in]{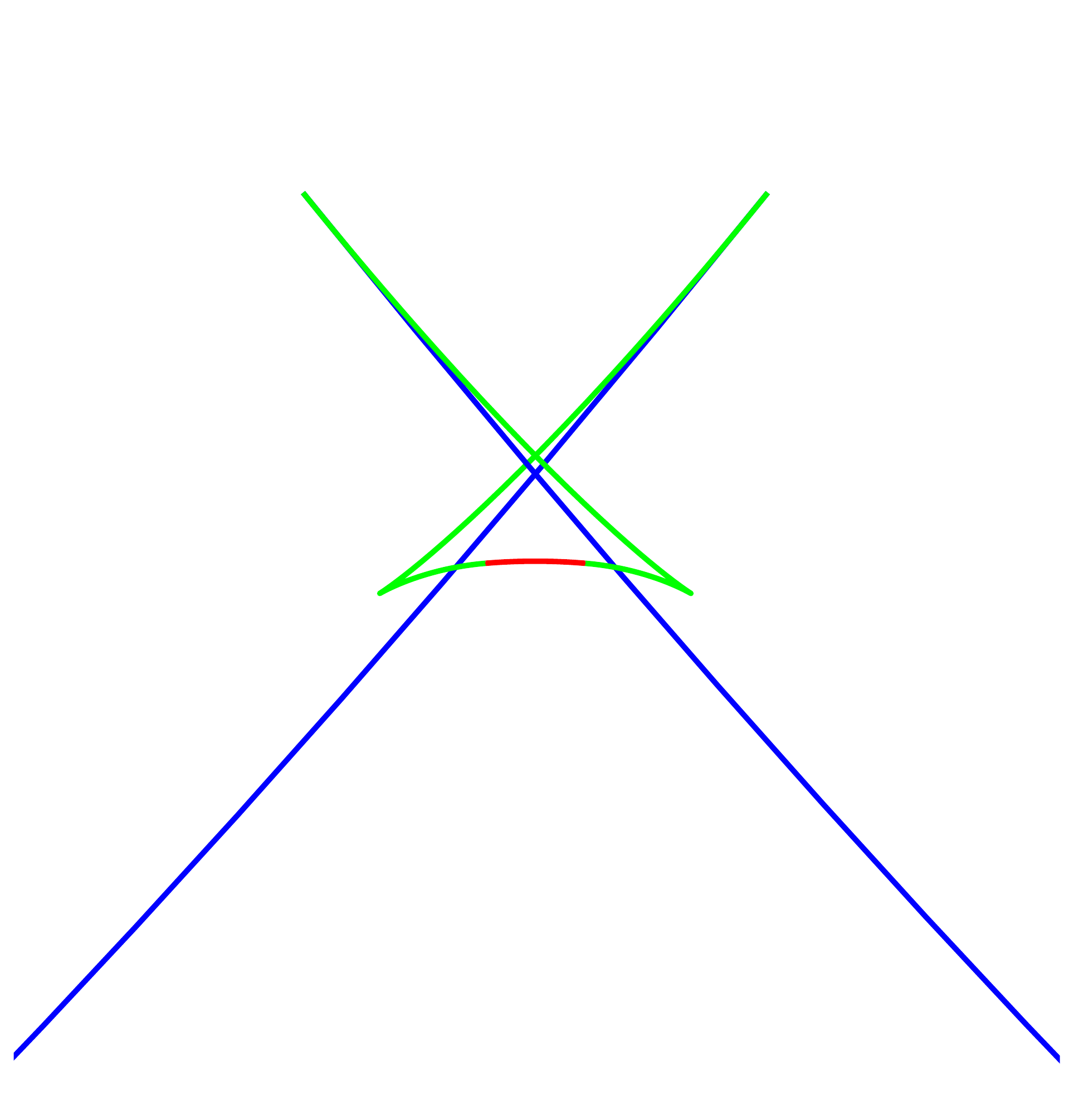}
\caption{$\F1^{t-s}(R_0^{s}v)$}\label{mm22}
\end{minipage}
\end{figure}

\begin{figure}[H]
\begin{minipage}[t]{0.5\linewidth}
\centering
\includegraphics[width=2.0in]{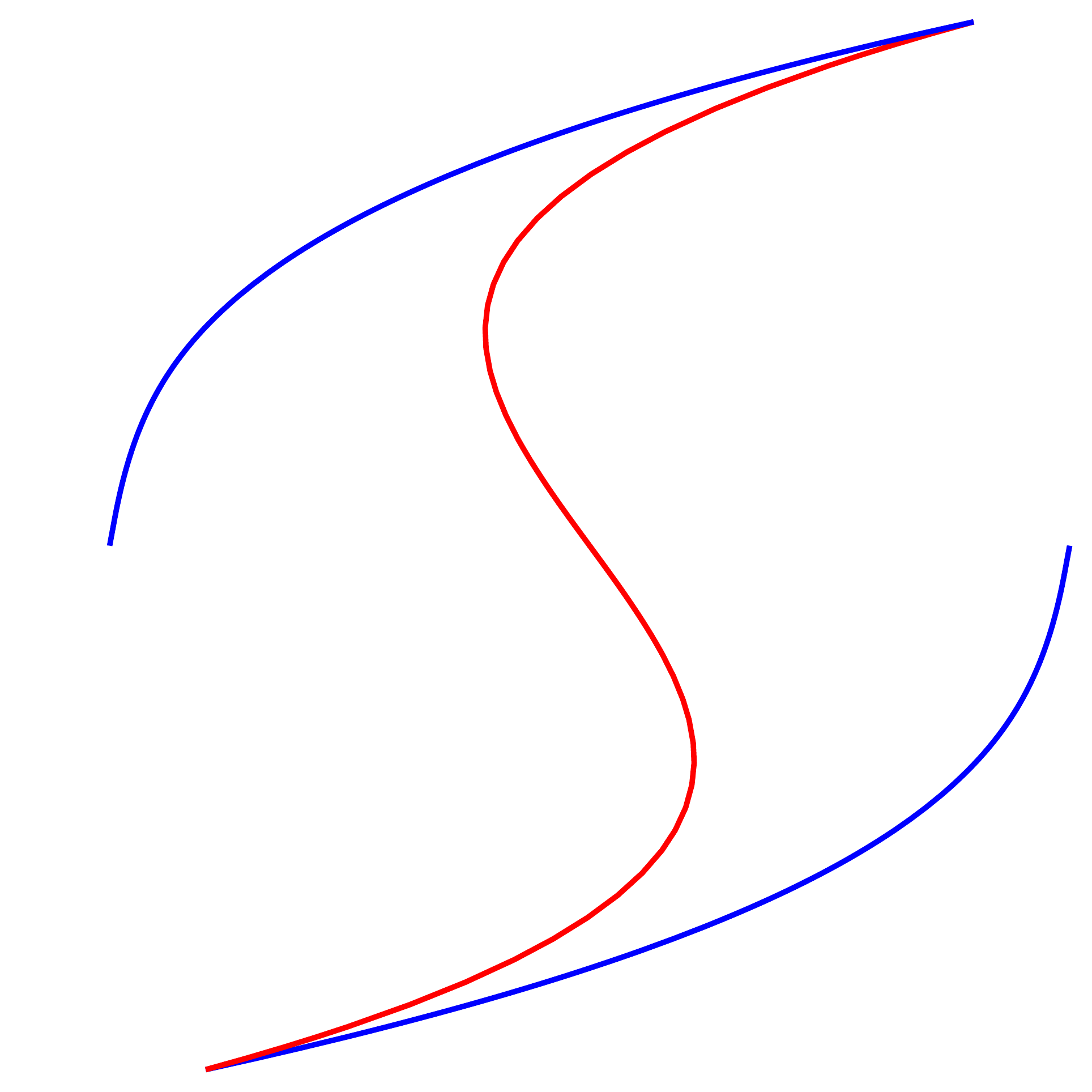}\label{lmm2}
\caption{$\vp^t(\p v)$}
\end{minipage}%
\begin{minipage}[t]{0.5\linewidth}
\centering
\includegraphics[width=2.0in]{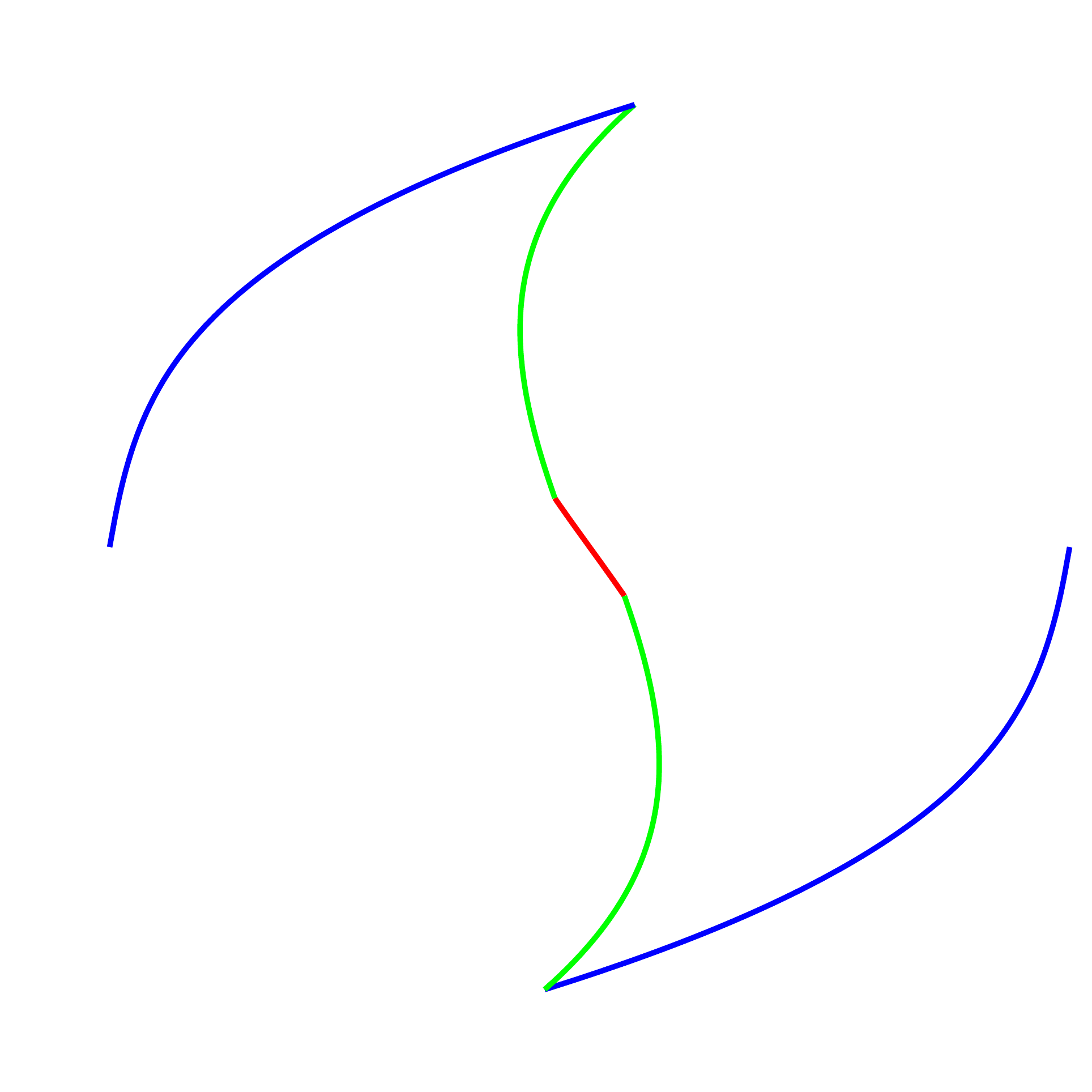}
\caption{$\vp^{t-s}(\p R_0^{s}v)$}\label{lmm22}
\end{minipage}
\end{figure}
As depicted in the wave front $\F1^t(v)$, the 1-step minmax $R_0^tv(x)$ has two shocks
$\chi_1^l(t)$ and $\chi_1^{r}(t)$. Both violate the entropy condition, hence the 2-step minmax
$R_s^t\circ R_0^sv(x)$ has new shocks $\chi_2^l(t)$ and $\chi_2^r(t)$ for $t>s$. Applying the same argument as before, by the limiting process
of iterated minmax, we get that the viscosity solution has two contact shocks. See the right of Figure \ref{shocktwo}.

\input{shocktwo.TpX}

 \bibliographystyle{alpha}
 \bibliography{bibliothese} 

\begin{thebibliography}{WEI13b}

\bibitem[CEL84]{GL}
M.~G. Crandall, L.~C. Evans, and P.-L. Lions.
\newblock Some properties of viscosity solutions of {H}amilton-{J}acobi
  equations.
\newblock {\em Trans. Amer. Math. Soc.}, 282(2):487--502, 1984.

\bibitem[Daf72]{DA}
Constantine~M. Dafermos.
\newblock Polygonal approximations of solutions of the initial value problem
  for a conservation law.
\newblock {\em J. Math. Anal. Appl.}, 38:33--41, 1972.

\bibitem[Eli87]{Elia}
Ya.~M. Eliashberg.
\newblock A theorem on the structure of wave fronts and its application in
  symplectic topology.
\newblock {\em Funktsional. Anal. i Prilozhen.}, 21(3):65--72, 96, 1987.

\bibitem[Hop65]{Hopf}
Eberhard Hopf.
\newblock Generalized solutions of non-linear equations of first order.
\newblock {\em J. Math. Mech.}, 14:951--973, 1965.

\bibitem[HR11]{HL}
Helge Holden and Nils~Henrik Risebro.
\newblock {\em Front tracking for hyperbolic conservation laws}, volume 152 of
  {\em Applied Mathematical Sciences}.
\newblock Springer, New York, 2011.
\newblock First softcover corrected printing of the 2002 original.

\bibitem[IK96]{Bif}
Shyuichi Izumiya and Georgios~T. Kossioris.
\newblock Formation of singularities for viscosity solutions of
  {H}amilton-{J}acobi equations.
\newblock In {\em Singularities and differential equations ({W}arsaw, 1993)},
  volume~33 of {\em Banach Center Publ.}, pages 127--148. Polish Acad. Sci.
  Inst. Math., Warsaw, 1996.

\bibitem[Lio82]{PL}
Pierre-Louis Lions.
\newblock {\em Generalized solutions of {H}amilton-{J}acobi equations},
  volume~69 of {\em Research Notes in Mathematics}.
\newblock Pitman (Advanced Publishing Program), Boston, Mass., 1982.

\bibitem[LR86]{HP}
P.-L. Lions and J.-C. Rochet.
\newblock Hopf formula and multitime {H}amilton-{J}acobi equations.
\newblock {\em Proc. Amer. Math. Soc.}, 96(1):79--84, 1986.

\bibitem[WEI11]{WQ}
Qiaoling WEI.
\newblock Subtleties of the minmax selector.
\newblock arxiv.org/abs/1112.5272, 2011.

\bibitem[WEI13a]{these}
Qiaoling WEI.
\newblock {\em Solutions de viscosit\'e des \'equations de Hamilton-Jacobi et
  minmax it\'er\'es}.
\newblock PhD thesis, Universit\'{e} de Paris 7, 2013.

\bibitem[WEI13b]{WQ2}
Qiaoling WEI.
\newblock Viscosity solution of {H}amilton-{J}acobi equation by a limiting
  minmax method.
\newblock 2013.

\end{thebibliography}
\end{document}